\documentclass[11pt]{article}
\usepackage{tikz}
\usepackage{amsmath,amsthm,amssymb,xcolor,multicol}

\usepackage{ifthen}
\usepackage{verbatim}

\usepackage[numbers]{natbib}

\newtheorem{theorem}{Theorem}[section]
\newtheorem{lemma}[theorem]{Lemma}
\newtheorem{proposition}[theorem]{Proposition}
\newtheorem{corollary}[theorem]{Corollary}
\newtheorem{definition}[theorem]{Definition}
\newtheorem{conjecture}[theorem]{Conjecture}

\newenvironment{remark}[1][Remark]{\begin{trivlist}
\item[\hskip \labelsep {\bfseries #1}]}{\end{trivlist}}

\newcommand{\bpi}[1]{\bar{\pi_{#1}}}
\newcommand{\barp}{\bar{p}}
\newcommand{\barx}{\bar{x}}
\newcommand{\bary}{\bar{y}}

\begin{document}

\author{Tom Denton}
\title{A Combinatorial Formula for Orthogonal Idempotents in the $0$-Hecke Algebra of the Symmetric Group}


\maketitle
\begin{abstract}

Building on the work of P.N. Norton, we give combinatorial formulae for two maximal decompositions of the identity into orthogonal idempotents in the $0$-Hecke algebra of the symmetric group, $\mathbb{C}H_0(S_N)$.  This construction is compatible with the branching from $S_{N-1}$ to $S_{N}$.


\end{abstract}

\section{Introduction}
\label{sec:in}

The $0$-Hecke algebra $\mathbb{C}H_0(S_N)$ for the symmetric group $S_N$ can be obtained as the Iwahori-Hecke algebra of the symmetric group $H_q(S_N)$ at $q=0$.  It can also be constructed as the algebra of the monoid generated by anti-sorting operators on permutations of $N$.

P. N. Norton described the full representation theory of $\mathbb{C}H_0(S_N)$ in \cite{norton79}:  In brief, there is a collection of $2^{N-1}$ simple representations indexed by subsets of the usual generating set for the symmetric group, in correspondence with collection of $2^{N-1}$ projective indecomposable modules.  Norton gave a construction for some elements generating these projective modules, but these elements were neither orthogonal nor idempotent.  While it was known that an orthogonal collection of idempotents to generate the indecomposable modules exists, there was no known formula for these elements.

Herein, we describe an explicit construction for two different families of orthogonal idempotents in $\mathbb{C}H_0(S_N)$, one for each of the two orientations of the Dynkin diagram for $S_N$.  The construction proceeds by creating a collection of $2^{N-1}$ \emph{demipotent} elements, which we call \emph{diagram demipotents}, each indexed by a copy of the Dynkin diagram with signs attached to each node.  These elements are demipotent in the sense that, for each element $X$, there exists some number $k\leq N-1$ such that $X^j$ is idempotent for all $j\geq k$.  The collection of idempotents thus obtained provides a maximal orthogonal decomposition of the identity.

An important feature of the $0$-Hecke algebra is that it is the monoid algebra of a $\mathcal{J}$-trivial monoid.  As a result, its representation theory is highly combinatorial.  This paper is part of an ongoing effort with Hivert, Schilling, and Thi\'ery \cite{Denton_Hivert_Schilling_Thiery.JTrivialMonoids} to characterize the representation theory of general $\mathcal{J}$-trivial monoids, continuing the work of \cite{norton79}, \cite{hivert09-2}, \cite{hivert09}.  This effort is part of a general trend to better understand the representation theory of finite semigroups.  See, for example, \cite{izhakian.2010.semigroupsrepresentationsoversemirings}, \cite{steinberg.2006.moebius}, \cite{steinberg.2008.moebiusii}, \cite{almeida_margolis_steinberg_volkov.2009}, \cite{pennell_putcha_renner.1997}, and for a general overview, \cite{ganyushkin09}.

The diagram demipotents obey a branching rule which compares well to the situation in \cite{okounkov96} in their `New Approach to the Representation Theory of the Symmetric Group.'  In their construction, the branching rule for $S_N$ is given primary importance, and yields a canonical basis for the irreducible modules for $S_N$ which pull back to bases for irreducible modules for $S_{N-M}$.  

Okounkov and Vershik further make extensive use of a maximal commutative algebra generated by the Jucys-Murphy elements.  In the $0$-Hecke algebra, their construction does not directly apply, because the deformation of Jucys-Murphy elements (which span a maximal commutative subalgebra of $\mathbb{C}S_N$) to the $0$-Hecke algebra no longer commute.  Instead, the idempotents obtained from the diagram demipotents play the role of the Jucys-Murphy elements, generating a commutative subalgebra of $\mathbb{C}H_0(S_N)$ and giving a natural decomposition into indecomposable modules, while the branching diagram describes the multiplicities of the irreducible modules.

The Okounkov-Vershik construction is well-known to extend to group algebras of general finite Coxeter groups (\cite{Ram97seminormalrepresentations}).  It remains to be seen whether our construction for orthogonal idempotents generalizes beyond type $A$.  However, the existence of a process for type $A$ gives hope that the Okounkov-Vershik process might extend to more general $0$-Hecke algebras of Coxeter groups.

Section~\ref{sec:bg} establishes notation and describes the relevant background necessary for the rest of the paper.  For further background information on the properties of the symmetric group, one can refer to the books of \cite{humphreys90} and \cite{stanley97}.
Section~\ref{sec:cand} gives the construction of the diagram demipotents.  
Section~\ref{sec:bra} describes the branching rule the diagram demipotents obey, and also establishes the Sibling Rivalry Lemma, which is useful in proving the main results, in Theorem~\ref{thm:main}.  Section~\ref{sec:nilp} establishes bounds on the power to which the diagram demipotents must be raised to obtain an idempotent.  
Finally, remaining questions are discussed in Section~\ref{sec:quest}.

\begin{remark}[Acknowledgements.]
\label{sec:ack}
This work was the result of an exploration suggested by Nicolas M. Thi\'ery; the notion of branching idempotents was suggested by Alain Lascoux.  Additionally, Florent Hivert gave useful insights into working with demipotents elements in an aperiodic monoid.  Thanks are also due to my advisor, Anne Schilling, as well as Chris Berg, Andrew Berget, Brant Jones, Steve Pon, and Qiang Wang for their helpful feedback.  This research was driven by computer exploration using the open-source mathematical software Sage, developed by ~\cite{sage} and its algebraic combinatorics features developed by the \cite{Sage-Combinat}, and in particular Daniel Bump and Mike Hansen who implemented the Iwahori-Hecke algebras.  For larger examples, the Semigroupe package developed by Jean-\'Eric Pin \cite{Semigroupe} was invaluable, saving perhaps weeks of computing time.   
\end{remark}

\section{Background and Notation}
\label{sec:bg}

Let $S_N$ be the symmetric group generated by the simple transpositions $s_i$ for $i\in I=\{1,\ldots, N-1\}$ which satisfy the following realtions:
\begin{itemize}
\item Reflection: $s_i^2=1$,
\item Commutation: $s_i s_j=s_j s_i$ for $|i-j|>1$,
\item Braid relation: $s_i s_{i+1} s_i=s_{i+1}s_i s_{i+1}$.
\end{itemize}

The relations between distinct generators are encoded in the \emph{Dynkin diagram} for $S_N$, which is a graph with one node for each generator $s_i$, and an edge between the pairs of nodes corresponding to generators $s_i$ and $s_{i+1}$ for each $i$.  Here, an edge encodes the braid relation, and generators whose nodes are not connected by an edge commute.  (See figure~\ref{fig:signedDiagram}.)

\begin{definition}
The {\bf $0$-Hecke monoid} $H_0(S_N)$ is generated by the collection $\pi_i$ for $i$ in the set $I=\{1, \ldots, N-1\}$ with relations:
\begin{itemize}
\item Idempotence: $\pi_i^2=\pi_i$,
\item Commutation: $\pi_i \pi_j=\pi_j \pi_i$ for $|i-j|>1$,
\item Braid Relation: $\pi_i \pi_{i+1} \pi_i=\pi_{i+1} \pi_i \pi_{i+1}$.
\end{itemize}
\end{definition}

The $0$-Hecke monoid can be realized combinatorially as the collection of anti-sorting operators on permutations of $N$.  For any permutation $\sigma$, $\pi_i\sigma=\sigma$ if $i+1$ comes before $i$ in the one-line notation for $\sigma$, and $\pi_i\sigma=s_i \sigma$ otherwise.  

Additionally, $\sigma\pi_i=\sigma s_i$ if the $i$th entry of $\sigma$ is less than the $i+1$th entry, and $\sigma\pi_i=\sigma$ otherwise.  (The left action of $\pi_i$ is on \emph{values}, and the right action is on \emph{positions}.)

\begin{definition}
The {\bf $0$-Hecke algebra} $\mathbb{C}H_0(S_N)$ is the monoid algebra of the $0$-Hecke monoid of the symmetric group.
\end{definition}

\begin{remark}[Words for $S_N$ and $H_0(S_N)$ Elements.]
The set $I=\{1,\ldots, N-1\}$ is called the \emph{index set} for the Dynkin diagram.  A \emph{word} is a sequence $(i_1, \ldots, i_k )$ of elements of the index set.  To any word $w$ we can associate a permutation $s_w=s_{i_1}\ldots s_{i_k}$ and an element of the $0$-Hecke monoid $\pi_w=\pi_{i_1}\cdots \pi_{i_k}$.  A word $w$ is \emph{reduced} if its length is minimal amongst words with permutation $s_w$.  The \emph{length} of a permutation $\sigma$ is equal to the length of a reduced word for $\sigma$.

Elements of the $0$-Hecke monoid are indexed by permutations:  Any reduced word $s=s_{i_1}\cdots s_{i_k}$ for a permutation $\sigma$ gives a reduced word in the $0$-Hecke monoid, $\pi_{i_1}\cdots \pi_{i_k}$.  Furthermore, given two reduced words $w$ and $v$ for a permutation $\sigma$, then $w$ is related to $v$ by a sequence of braid and commutation relations.  These relations still hold in the $0$-Hecke monoid, so $\pi_w=\pi_v$.  

From this, we can see that the $0$-Hecke monoid has $N!$ elements, and that the $0$-Hecke algebra has dimension $N!$ as a vector space.  Additionally, the length of a permutation is the same as the length of the associated $H_0(S_N)$ element.

We can obtain a \emph{parabolic subgroup} (resp. submonoid, subalgebra) by considering the object whose generators are indexed by a subset $J\subset I$, retaining the original relations.  The Dynkin diagram of the corresponding object is obtained by deleting the relevant nodes and connecting edges from the original Dynkin diagram.  Every parabolic subgroup of $S_N$ contains a unique longest element, being an element whose length is maximal amongst all elements of the subgroup.  We denote the longest element in the parabolic sub-monoid of $H_0(S_N)$ with generators indexed by $J\subset I$ by $w_J^+$, and use $\hat{J}$ to denote the complement of $J$ in $I$.  For example, in $H_0(S_8)$ with $J=\{1,2,6\}$, then $w_J^+=\pi_{1216}$, and $w_{\hat{J}}^+=\pi_{3453437}$.

\begin{definition}
An element $x$ of a monoid or algebra is \emph{demipotent} if there exists some $k$ such that $x^\omega := x^k=x^{k+1}$.  A monoid is \emph{aperiodic} if every element is demipotent.
\end{definition}

The $0$-Hecke monoid is aperiodic.  Namely, for any element $x\in H_0(S_N)$, let: 
\[
J(x)=\{ i\in I \mid \text{ s.t. $i$ appears in some reduced word for $x$} \}.
\]
This set is well defined because if $i$ appears in some reduced word for $x$, then it appears in every reduced word for $x$.  Then $x^\omega=w_{J(x)}^+$.
\end{remark}

\begin{remark}[The Algebra Automorphism $\Psi$ of $\mathbb{C}H_0(S_N)$.]
$\mathbb{C}H_0(S_N)$ is alternatively generated as an algebra by elements $\bpi{i}:=(1-\pi_i)$, which satisfy the same relations as the $\pi_i$ generators.  There is a unique automorphism $\Psi$ of $\mathbb{C}H_0(S_N)$ defined by sending $\pi_i \rightarrow (1-\pi_i)$.

For any longest element $w_J^+$, the image $\Psi(w_J^+)$ is a longest element in the $(1-\pi_i)$ generators; this element is denoted $w_J^-$.
\end{remark}

\begin{remark}[The Dynkin diagram Automorphism of $\mathbb{C}H_0(S_N)$.]
A Dynkin diagram automorphism is a graph automorphism of the underlying graph.  For the Dynkin diagram of $S_N$, there is exactly one non-trivial automorphism, sending the node $i$ to $N-i$.

This diagram automorphism induces an automorphism of the symmetric group, sending the generator $s_i\rightarrow s_{N-i}$ and extending multiplicatively.  Similarly, there is an automorphism of the $0$-Hecke monoid sending the generator $\pi_i\rightarrow \pi_{N-i}$ and extending multiplicatively.
\end{remark}

\begin{remark}[Bruhat Order.]
The \emph{(left) weak order} on the set of permutations is defined by the relation $\sigma \leq_L \tau$ if there exist reduced words $v, w$ such that $\sigma=s_v, \tau=s_w$, and $v$ is a prefix of $w$ in the sense that $w=v_1,v_2, \ldots, v_j, w_j+1, \ldots, w_k$.  The right weak order is defined analogously, where $v$ must appear as a suffix of $w$.

The left weak order also exists on the set of $0$-Hecke monoid elements, with exactly the same definition.  Indeed, $s_v\leq_L s_w$ if and only if $\pi_v\leq_L \pi_w$.

For a permutation $\sigma$, we say that $i$ is a \emph{(left) descent} of $\sigma$ if $s_i\sigma \leq_L \sigma$.  We can define a descent in the same way for any element $\pi_w$ of the $0$-Hecke monoid.  We write $D_L(\sigma)$ and $D_L(\pi_w)$ for the set of all descents of $\sigma$ and $m$ respectively.  Right descents are defined analogously, and are denoted $D_R(\sigma)$ and $D_R(\pi_w)$, respectively.

It is well known that $i$ is a left descent of $\sigma$ if and only if there exists a reduced word $w$ for $\sigma$ with $w_1=i$.  As a consequence, if $D_L(\pi_w)=J$, then $w_J^+\pi_w=\pi_w$.  Likewise, $i$ is a right descent if and only if there exists a reduced word for $\sigma$ ending in $i$, and if $D_R(\pi_w)=J$, then $\pi_ww_J^+=\pi_w$.

\emph{Bruhat order} is defined by the relation $\sigma\leq \tau$ if there exist reduced words $v$ and $w$ such that $s_v=\sigma$ and $s_w=\tau$ and $v$ appears as a subword of $w$.  For example, $13$ appears as a subword of $123$, so $s_{12}\leq s_{123}$ in strong Bruhat order.
\end{remark}

\begin{remark}[Representation Theory]
The representation theory of $\mathbb{C}H_0(S_N)$ was described in \cite{norton79} and expanded to generic finite Coxeter groups in \cite{carter198689}.  A more general approach to the representation theory can be taken by approaching the $0$-Hecke algebra as a monoid algebra, as per  \cite{ganyushkin09}.  The main results are reproduced here for ease of reference.

For any subset $J\subset I$, let $\lambda_J$ denote the one-dimensional representation of $H$ defined by the action of the generators:
\begin{equation*}
    \lambda_J(\pi_i) = \begin{cases}
      0 & \text{if $i\in J$},\\
      1 & \text{if $i\not\in J.$}
    \end{cases}
\end{equation*}
The $\lambda_J$ are $2^{N-1}$ non-isomorphic representations, all one-dimensional and thus simple.  In fact, these are all of the simple representations of $\mathbb{C}H_0(S_N)$.

\begin{definition}
For each $i\in I$, define the \emph{evaluation maps} $\Phi_i^+$ and $\Phi_i^+$ on generators by:
\begin{eqnarray*}
\Phi_N^+ &:& \mathbb{C}H_0(W) \rightarrow \mathbb{C}H_0(W_{I\setminus \{i\}}) \\
\Phi_N^+(\pi_i) &=&     \begin{cases}
      1          & \text{if $i=N$,}\\
      \pi_i & \text{if $i \neq N$.}
    \end{cases}\\
\Phi_N^- &:& \mathbb{C}H_0(W) \rightarrow \mathbb{C}H_0(W_{I\setminus \{i\}}) \\
\Phi_N^-(\pi_i) &=&     \begin{cases}
      0          & \text{if $i=N$,}\\
      \pi_i & \text{if $i \neq N$.}
    \end{cases}
\end{eqnarray*}
\end{definition}
One can easily check that these maps extend to algebra morphisms from $H_0(W)\rightarrow H_0(W_{I\setminus i})$.  For any $J$, define $\Phi_J^+$ as the composition of the maps $\Phi_i^+$ for $i\in J$, and define $\Phi_J^-$ analogously.  Then the simple representations of $H_0(W)$ are given by the maps $\lambda_J = \Phi_J^+ \circ \Phi_{\hat{J}}^-$, where $\hat{J}=I\setminus J$.

The map $\Phi_J^+$ is also known as the parabolic map~\cite{Billey19991}, which sends an element $x$ to an element $y$ such that $y$ is the longest element less than $x$ in Bruhat order in the parabolic submonoid with generators indexed by $J$.

The nilpotent radical $\mathcal{N}$ in $\mathbb{C}H_0(S_N)$ is spanned by elements of the form $x-w_{J(x)}^+$, where $x$ is an element of the monoid $H_0(S_N)$, and $w_{J(x)}^+$ is the longest element in the parabolic submonoid whose generators are exactly the generators in any given reduced word for $x$.  This element $w_{J(x)}^+$ is idempotent.  If $y$ is already idempotent, then $y=w_{J(y)}^+$, and so $y-w_{J(y)}^+=0$ contributes nothing to $\mathcal{N}$.  However, all other elements $x-w_{J(x)}^+$ for $x$ not idempotent are linearly independent, and thus give a basis of $\mathcal{N}$.

Norton further showed that 
\begin{equation*}
\mathbb{C}H_0(S_N)=\bigoplus_{J\subset I} H_0(S_N)w_J^-w_{\hat{J}}^+
\end{equation*}
is a direct sum decomposition of $\mathbb{C}H_0(S_N)$ into indecomposable left ideals.

\begin{theorem}[Norton, 1979]
\label{thm:norton}
Let $\{p_J | J\subset I\}$ be a set of mutually orthogonal primitive idempotents with $p_J \in \mathbb{C}H_0(S_N)w_J^-w_{\hat{J}}^+$ for all $J\subset I$ such that $\sum_{J\subset I} p_J = 1$.  

Then $\mathbb{C}H_0(S_N)w_J^-w_{\hat{J}}^+=\mathbb{C}H_0(S_N)p_J$, and if $\mathcal{N}$ is the nilpotent radical of $\mathbb{C}H_0(S_N)$, $\mathcal{N}w_J^-w_{\hat{J}}^+=\mathcal{N}p_J$ is the unique maximal left ideal of $\mathbb{C}H_0(S_N)p_J$, and $\mathbb{C}H_0(S_N)p_J/\mathcal{N}p_J$ affords the representation $\lambda_J$.

Finally, the commutative algebra 
\[
\mathbb{C}H_0(S_N)/\mathcal{N} \stackrel{~}{=} \bigoplus_{J\subset I}\mathbb{C}H_0(S_N)p_J/\mathcal{N}p_J 
\stackrel{~}{=} \mathbb{C}^{2^{N-1}}.
\]
\end{theorem}

The elements $w_J^-w_{\hat{J}}^+$ are neiter orthogonal nor idempotent; the proof of Norton's theorem is non-constructive, and does not give a formula for the idempotents.
\end{remark}

\section{Diagram Demipotents}
\label{sec:cand}

The elements $\pi_i$ and $(1-\pi_i)$ are idempotent.  There are actually $2^{N-1}$ idempotents in $H_0(S_N)$, namely the elements $w_J^+$ for any $J\subset I$.  These idempotents are clearly not orthogonal, though.  The goal of this paper is to give a formula for a collection of \emph{orthogonal} idempotents in $\mathbb{C}H_0(S_N)$.  

For our purposes, it will be convenient to index subsets of the index set $I$ (and thus also simple and projective representations) by \emph{signed diagrams.}

\begin{definition}
A \emph{signed diagram} is a Dynkin diagram in which each vertex is labeled with a $+$ or $-$.  
\end{definition}

Figure~\ref{fig:signedDiagram} depicts a signed diagram for type $A_7$, corresponding to $H_0(S_8)$.  For brevity, a diagram can be written as just a string of signs.  For example, the signed diagram in the Figure is written $++---+-$.

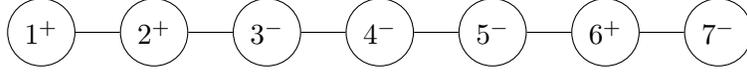
\begin{figure}
\centering
\begin{tikzpicture}[scale=.75]
\node (one) at ( 0,0) [circle,draw] {$1^+$};
\node (two) at ( 2,0) [circle,draw] {$2^+$};
\node (three) at ( 4,0) [circle,draw] {$3^-$};
\node (four) at ( 6,0) [circle,draw] {$4^-$};
\node (five) at ( 8,0) [circle,draw] {$5^-$};
\node (six) at ( 10,0) [circle,draw] {$6^+$};
\node (seven) at ( 12,0) [circle,draw] {$7^-$};

\draw [-] (one.east) -- (two.west);
\draw [-] (two.east) -- (three.west);
\draw [-] (three.east) -- (four.west);
\draw [-] (four.east) -- (five.west);
\draw [-] (five.east) -- (six.west);
\draw [-] (six.east) -- (seven.west);
\end{tikzpicture}
\caption{A signed Dynkin diagram for $S_8$.}
\label{fig:signedDiagram}
\end{figure}

We now construct a \emph{diagram demipotent} corresponding to each signed diagram.  Let $P$ be a composition of the index set $I$ obtained from a signed diagram $D$ by grouping together sets of adjacent pluses and minuses.  For the diagram in Figure~\ref{fig:signedDiagram}, we would have $P = \{ \{1,2\} , \{3,4,5\}, \{6,7\} \}$.  Let $P_k$ denote the $k$th subset in $P$.  For each $P_k$, let $w_{P_k}^{sgn(k)}$ be the longest element of the parabolic sub-monoid associated to the index set $P_k$, constructed with the generators $\pi_i$ if $sgn(k)=+$ and constructed with the $(1-\pi_i)$ generators if $sgn(k)=-$.

\begin{definition}
\label{def:ddemipotents}
Let $D$ be a signed diagram with associated composition $P= P_1 \cup \cdots \cup P_m $.  Set: 
\begin{eqnarray*}
L_D&=&w_{P_1}^{sgn(1)}w_{P_2}^{sgn(2)}\cdots w_{P_m}^{sgn(m)},\text{ and }\\
R_D&=&w_{P_m}^{sgn(m)}w_{P_{m-1}}^{sgn(m-1)}\cdots w_{P_1}^{sgn(1)}.
\end{eqnarray*}
The \emph{diagram demipotent} $C_D$ associated to the signed diagram $D$ is then $L_DR_D$.  The \emph{opposite diagram demipotent} $C_D'$ is $R_DL_D$.
\end{definition}

Thus, the diagram demipotent for the diagram in Figure~\ref{fig:signedDiagram} is 
\[
\pi_{121}^+\pi_{345343}^-\pi_{6}^+\pi_{7}^- \pi_{6}^+\pi_{345343}^-\pi_{121}^+.
\]

It is not immediately obvious that these elements are demipotent; this is a direct result of Lemma~\ref{thm:sibRiv}.

For $N=1$, there is only the empty diagram, and the diagram demipotent is just the identity.

For $N=2$, there are two diagrams, $+$ and $-$, and the two diagram demipotents are $\pi_1$ and $1-\pi_1$ respectively.  Notice that these form a decomposition of the identity, as $\pi_i+(1-\pi_i)=1$.

For $N=3$, we have the following list of diagram demipotents.  The first column gives the diagram, the second gives the element written as a product, and the third expands the element as a sum.  For brevity, words in the $\pi_i$ or $\bpi{i}$ generators are written as strings in the subscripts.  Thus, $\pi_1\pi_2$ is abbreviated to $\pi_{12}$.

\begin{equation*}
\begin{array}[b]{|c|c|c|}
D & C_D & \text{Expanded Demipotent} \\ \hline
++ & \pi_{121} 			& \pi_{121} \\
+- & \pi_1\bpi{2}\pi_1 		& \pi_1 - \pi_{121} \\
-+ & \bpi{1}\pi_2\bpi{1} 	& \pi_2 - \pi_{12} - \pi_{21} + \pi_{121} \\
- -& \bpi{121} 			& 1 - \pi_1 - \pi_2 + \pi_{12} + \pi_{21} - \pi_{121}\\
\end{array}
\end{equation*}

Observations.

\begin{itemize}
\item The idempotent $\bpi{121}$ is an alternating sum over the monoid.  This is a general phenomenon: By \cite{norton79}, $w_J^-$ is the length-alternating signed sum over the elements of the parabolic sub-monoid with generators indexed by $J$.

\item The shortest element in each expanded sum is an idempotent in the monoid with $\pi_i$ generators; this is also a general phenomenon.  The shortest term is just the product of longest elements in nonadjacent parabolic sub-monoids, and is thus idempotent.  Then the shortest term of $C_D$ is $\pi_J^+$, where $J$ is the set of nodes in $D$ marked with a $+$.  Each diagram yields a different leading term, so we can immediately see that the $2^{N-1}$ idempotents in the monoid appear as a leading term for exactly one of the diagram demipotents, and that they are linearly independent.

\item For many purposes, one only needs to explicitly compute half of the list of diagram demipotents; the other half can be obtained via the automorphism $\Psi$.  A given diagram demipotent $x$ is orthogonal to $\Psi(x)$, since one has left and right $\pi_1$ descents, and the other has left and right $\bpi{1}$ descents, and $\pi_1 \bpi{1}=0$.

\item The diagram demipotents are fixed under the automorphism determined by $\pi_{\sigma}\rightarrow \pi_{sigma^{-1}}$.  In particular, $L_D$ is the reverse of $R_D$, and $C_D$ can be expressed as a palindrome in the alphabet $\{\pi_i, \bpi{i}\}$.

\item The diagram demipotents $C^D$ and $C^E$ for $D\neq E$ do not necessarily commute.  Non-commuting demipotents first arise with $N=6$.  However, the idempotents obtained from the demipotents are orthogonal and do commute.

\item It should also be noted that these demipotents (and the resulting idempotents) are not in the projective modules constructed by Norton, but generate projective modules isomorphic to Norton's.

\item The diagram demipotents $C_D$ listed here are not fixed under the automorphism induced by the Dynkin diagram automorphism.  In particular, the `opposite' diagram demipotents $C_D'=R_DL_D$ really are different elements of the algebra, and yield an equally valid but different set of orthogonal idempotents.  For purposes of comparison, the diagram demipotents for the reversed Dynkin diagram are listed below for $N=3$.

\begin{equation*}
\begin{array}[b]{|c|c|c|}
D & C_D' & \text{Expanded Demipotent} \\ \hline
++ & \pi_{212} 			& \pi_{212} \\
+- & \pi_2\bpi{1}\pi_2 		& \pi_2 - \pi_{212} \\
-+ & \bpi{2}\pi_1\bpi{2} & \pi_1 - \pi_{12} - \pi_{21} + \pi_{212} \\
- -& \bpi{212} 	& 1 - \pi_1 - \pi_2 + \pi_{12} + \pi_{21} - \pi_{212}\\
\end{array}
\end{equation*}
\end{itemize}

For $N\leq 4$, the diagram demipotents are actually idempotent and orthogonal.  For larger $N$, raising the diagram demipotent to a sufficiently large power yields an idempotent (see below~\ref{thm:main}); in other words, the diagram demipotents are demipotent.  The power that an diagram demipotent must be raised to in order to obtain an actual idempotent is called its \emph{nilpotence degree}.

For $N=5$, two of the diagram demipotents need to be squared to obtain an idempotent.  For $N=6$, eight elements must be squared.  For $N=7$, there are four elements that must be cubed, and many others must be squared.  Some pretty good upper bounds on the nilpotence degree of the diagram demipotents are given in Section~\ref{sec:nilp}.  As a preview, for $N>4$ the nilpotence degree is always $\leq N-3$, and conditions on the diagram can often greatly reduce this bound.

As an alternative to raising the demipotent to some power, we can express the idempotents as a product of diagram demipotents for smaller diagrams.  Let $D_k$ be the signed diagram obtained by taking only the first $k$ nodes of $D$.  Then, as we will see, the idempotents can also be expressed as the product $C_{D_1}C_{D_2}C_{D_3}\cdots C_{D_{N-1}=D}$.

\begin{remark}[Left Weak Order.]
Let $m$ be a standard basis element of the $0$-Hecke algebra in the $\pi_i$ basis.  Then for any $i\in D_L(m)$, $\pi_i m=m$, and for any $i\not \in D_L(m)$ then $\pi_im\geq_L m$, in left weak order.  This is an adaptation of a standard fact in the theory of Coxeter groups to the $0$-Hecke setting.
\end{remark}

\begin{corollary}[Diagram Demipotent Triangularity]
\label{cor:triangularity}
Let $C_D$ be a diagram demipotent and $m$ an element of the $0$-Hecke monoid in the $\pi_i$ generators.  Then $C_Dm = \lambda m + x$, where $x$ is an element of $H_0(S_N)$ spanned by monoid elements lower in left weak order than $m$, and $\lambda \in \{0,1\}$.  Furthermore, $\lambda=1$ if and only if $Des(m)$ is exactly the set of nodes in $D$ marked with pluses.
\end{corollary}

\begin{proof}
The diagram demipotent $C_D$ is a product of $\pi_i$'s and $(1-\pi_i)$'s.
\end{proof}

\begin{proposition}
Each diagram demipotent is the sum of a non-zero idempotent part and a nilpotent part.  That is, all eigenvalues of a diagram demipotent are either $1$ or $0$.
\end{proposition}

\begin{proof}
Assign a total ordering to the basis of $H_0(S_N)$ in the $\pi_i$ generators that respects the Bruhat order.  Then by Corollary~\ref{cor:triangularity}, the matrix $M_D$ of any diagram demipotent $C_D$ is lower triangular, and each diagonal entry of $M_D$ is either one or zero.  A lower triangular matrix with diagonal entries in $\{0,1\}$ has eigenvalues in $\{0,1\}$; thus $C_D$ is the sum of an idempotent and a nilpotent part.

To show that the idempotent part is non-zero, consider any element $m$ of the monoid such that $Des(m)$ is exactly the set of nodes in $D$ marked with pluses.  Then $C_Dm = m + x$ shows that $C_D$ has a $1$ on the diagonal, and thus has $1$ as an eigenvalue.  Then the idempotent part of $C_D$ is non-zero.  (This argument still works if $D$ has no plusses, since the associated diagram demipotent fixes the identity.)
\end{proof}

\section{Branching}
\label{sec:bra}

There is a convenient and useful branching of the diagram demipotents for $H_0(S_N)$ into diagram demipotents for $H_0(S_{N+1})$.

\begin{lemma}
\label{lem:longWords}
Let $J=\{i,i+1,\ldots,N-1\}$
Then $w_J^{+}\pi_N w_J^{+}$ is the longest element in the generators $i$ through $N$.  Likewise, $w_J^{+}\pi_{i-1}w_J^{+}$ is the longest element in the generators $i-1$ through $N-1$.  Similar statements hold for $w_J^{-}\bpi{N} w_J^{-}$ and $w_J^{-}\bpi{i-1} w_J^{-}$.
\end{lemma}

\begin{proof}
Let $K=[i,i+1,\ldots,N-1]$.

The lexicographically minimal reduced word for the longest element in consecutive generators $1$ through $k$ is obtained by concatenating the ascending sequences $\pi_1 \pi_2 \ldots \pi_{k-i}$ for all $0<i<k$.
For example, the longest element in generators $1$ through $4$ is $\pi_{1234123121}$.

Now form the product $m=w_J^+ \pi_N w_J^+$, for example $\pi_{1234123121}\pi_5\pi_{1234123121}$.  This contains a reduced word for $w_Y^+$ as a subword, and is thus greater than or equal to $X$ in the (strong) Bruhat Order.  But since $w_Y^+$ is the longest element in the given generators, $m$ and $w_Y^+$ must be equal.

For the second statement, apply the same methods using the lexicographically maximal word for the longest elements.

The analogous statement follows directly by applying the automorphism $\Psi$.
\end{proof}


Recall that each diagram demipotent $C_D$ is the product of two elements $L_D$ and $R_D$.  For a signed diagram $D$, let $D+$ denote the diagram with an extra $+$ adjoined at the end.  Define $D-$ analogously.  

\begin{corollary}
Let $C_D=L_DR_D$ be the diagram demipotent associated to the signed diagram $D$ for $S_N$.  Then $C_{D+}=L_D \pi_N R_D$ and $C_{D-}=L_D \bpi{N} R_D$.  In particular, $C_{D+} + C_{D-} = C_D$.  Finally, the sum of all diagram demipotents for $H_0(S_N)$ is the identity.
\end{corollary}

\begin{proof}
The identities $C_{D+}=L_D \pi_N R_D$ and $C_{D-}=L_D \bpi{N} R_D$ are consequences of Lemma~\ref{lem:longWords}, and the identity $C_{D+} + C_{D-} = C_D$ follows directly.

To show that the sum of all diagram demipotents for fixed $N$ is the identity, recall that the diagram demipotent for the empty diagram is the identity, then apply the identity $C_{D+} + C_{D-} = C_D$ repeatedly.
\end{proof}

Next we have a key lemma for proving many of the remaining results in this paper:

\begin{lemma}[Sibling Rivalry]
\label{thm:sibRiv}
Sibling diagram demipotents commute and are orthogonal: $C_{D-}C_{D+}=C_{D+}C_{D-}=0$.  
Equivalently, 
\[
C_{D}C_{D+}=C_{D+}C_{D}=C_{D+}^2 
\text{ and } C_{D}C_{D-}=C_{D-}C_{D}=C_{D-}^2.
\]
\end{lemma}

\begin{proof}

\begin{figure}
\centering
\begin{tikzpicture}[scale=.5]
\node (epsilon) at ( 2,6) [circle,draw] {$r$};

\node (phi) at ( 0,4) [circle,draw] {$q$};

\node (p) at ( -3, 2 )  [circle,draw] {$p$};
\node (bp) at ( 3, 2 )  [circle,draw] {$\barp$};

\node (x) at ( -5, 0 )  [circle,draw] {$x$};
\node (y) at ( -1, 0 )  [circle,draw] {$y$};

\node (bx) at ( 1, 0 )  [circle,draw] {$\barx$};
\node (by) at ( 5, 0 )  [circle,draw] {$\bary$};

\draw [-] (epsilon.west) -- (phi.north) node [above,midway]  {$+$};

\draw [-] (phi.west) -- (p.north) node [above,midway]  {$+$};
\draw [-] (phi.east) -- (bp.north) node  [above,midway] {$-$};

\draw [-] (p.west) -- (x.north) node  [above,midway] {$+$};
\draw [-] (p.east) -- (y.north) node [above,midway] {$-$};

\draw [-] (bp.west) -- (bx.north) node [above,midway] {$+$};
\draw [-] (bp.east) -- (by.north) node [above,midway] {$-$};
\end{tikzpicture}
\caption{Relationship of Elements in the Proof of the Sibling Rivalry Lemma.}
\label{fig:sibRivalry}
\end{figure}
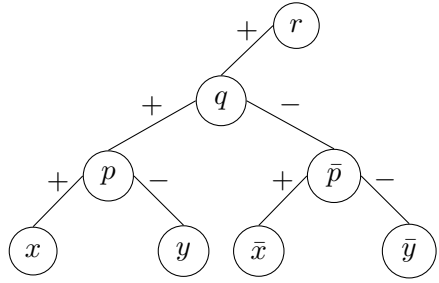

We proceed by induction, using two levels of branching.  Thus, we want to show the orthogonality of two diagram demipotents $x$ and $y$ which are branched from a parent $p$ and grandparent $q$.  Without loss of generality, let $q$ be the positive child of an element $r$.  Call $q$'s other child $\bar{p}$, which in turn has children $\bar{x}$ and $\bar{y}$.  The relations between the elements is summarized in Figure~\ref{fig:sibRivalry}.

The goal, then, is to prove that $yx=0$ and $\bary \barx=0$.  Since $p=x+y$, we have that $yx=(p-x)x=px-x^2$.  Thus, we can equivalently go about proving that $px=x^2$ or $py=y^2$.  It will be easier to show $px=x^2$.  We will also show that $\barp\barx = \barx^2$.  Once this is done, we will have proven the result for diagrams ending in $+++$, $++-$, $+-+$, and $+--$.  By applying the automorphism $\Psi$, we obtain the result for the other four cases.  

One can obtain the reverse equalities $xy=0, \barx\barp=0,$ and so on, either by performing equivalent computations, or else by another use of the $\Psi$ automorphism.  For the latter, suppose that we know $C_{D+}C_{D-}=0$ for arbitrary $D$.  Then applying $\Psi$ to this equation gives $C_{\hat{D}-}C_{\hat{D}+}=0$, where $\hat{D}$ is the signed diagram $D$ with all signs reversed.  Since $D$ was arbitrary, $\hat{D}$ is also arbitrary, so $C_{D-}C_{D+}=0$ for arbitrary $D$.

Let $r=LR$, dropping the $D$ subscript for convenience, generated with $i$ in the index set $I$.  Let the two new generators be $\pi_a, \pi_b$ and $\pi_c$.  Notice that $\pi_b$, $\bpi{b}$, $\pi_c$, and $\bpi{c}$ all commute with $L$ and $R$.  

The inductive hypothesis tells us that $pq=q p=p^2$ and $\barp q=q \barp=\barp^2$.  We also have the following identities:
\begin{itemize}
\item $q = L \pi_a R$,
\item $p = L \pi_a \pi_b \pi_a R = \pi_b q \pi_b$,
\item $x = L \pi_{aba} \pi_c \pi_{aba} R = \pi_{cbc} q \pi_{cbc}$,
\item $pq = q \pi_b q \pi_b = p^2 = \pi_b q \pi_b q \pi_b$.
\end{itemize}

Then we compute directly:
\begin{eqnarray*}
px & = & \pi_b q \pi_b \pi_{cbc} q \pi_{cbc} \\
   & = & \pi_b q \pi_{cbc} q \pi_{cbc} \\
   & = & \pi_{bc} (q \pi_{b} q \pi_b) \pi_{cbc} \\
   & = & \pi_{bc} ( \pi_b q \pi_{b} q \pi_b) \pi_{cbc} \\
   & = & \pi_{bcb} ( q \pi_{b} q ) \pi_{cbc} \\
   & = & \pi_{cbc} ( q \pi_{cbc} q ) \pi_{cbc} \\
   & = & x^2.
\end{eqnarray*}

To complete the proof, we need to show that $\barp\barx = \barx^2$.  To do so, we use the following identities:
\begin{itemize}
\item $q = L \pi_a R$,
\item $\barp = L \pi_a (1-\pi_b) \pi_a R$,
\item $\barx = L \pi_a (1-\pi_b) \pi_c (1-\pi_b) \pi_a R$.
\end{itemize}

Then we expand the following equation:
\[
\barp \barx = L \pi_a (1-\pi_b) \pi_a R L \pi_a (1-\pi_b) \pi_c (1-\pi_b) \pi_a R.
\]
We expand this as follows:
\[
\barp \barx = q^2 \pi_c 
- q \barp \pi_c 
- q \pi_c \barp 
+ q \pi_c \barp \pi_c 
- \barp q \pi_c  
+ \barp^2 \pi_c  
+ \barp \pi_c \barp
- \barp \pi_c \barp \pi_c.
\]

Meanwhile, 
\begin{eqnarray*}
\barx & = & L ( \pi_{ac} - \pi_{abca} - \pi_{acba} + \pi_{abcba} ) R \\
  & = & \pi_cq - \barp\pi_c - \pi_c \barp + \pi_c \barp \pi_c 
\end{eqnarray*}

Expanding $\barx^2$ in terms of $\barp$ and $q$ is a lengthy but straightforward calculation, which yields:
\begin{eqnarray*}
\barx^2 & = & q^2 \pi_c 
- q \barp \pi_c 
- q \pi_c \barp 
+ q \pi_c \barp \pi_c 
- \barp q \pi_c  
+ \barp^2 \pi_c  
+ \barp \pi_c \barp
- \barp \pi_c \barp \pi_c \\
& = & \barp \barx
\end{eqnarray*}

This completes the proof of the lemma.
\end{proof}


\begin{corollary}
The diagram demipotents $C_D$ are demipotent.
\end{corollary}

This follows immediately by induction: if $C_D^k=C_D^{k+1}$, then $C_{D+}C_D^k=C_{D+}C_D^{k+1}$, and by sibling rivalry, $C_{D+}^{k+1}=C_{D+}^{k+2}$.

Now we can say a bit more about the structure of the diagram demipotents.  

\begin{proposition}
\label{thm:algebra}
Let $p=C_D, x=C_{D+}, y=C_{D-}$, so $p=x+y$ and $xy=0$.  Let $v$ be an element of $H$.  Furthermore, let $p, x$, and $y$ have abstract Jordan decomposition $p=p_i+p_n$, $x=x_i+x_n$, $y=y_i+y_n$, with $p_ip_n=p_np_i$ and $p_i^2=p_i$, $p_n^k=0$ for some $k$, and similar relations for the Jordan decompositions of $x$ and $y$.

Then we have the following relations:
\begin{enumerate}
\item If there exists $k$ such that $p^kv=0$, then $x^{k+1}v=y^{k+1}v=0$.
\item If $pv=v$, then $x(x-1)v=0$
\item If $(x-1)^kv=0$, then $(x-1)v=0$
\item If $pv=v$ and $x^kv=0$ for some $k$, then $yv=v$.
\item If $xv=v$, then $yv=0$ and $pv=v$.
\item Let $u^x_i$ be a basis of the $1$-space of $x$, so that $xu^x_i=u^x_i$, $yu^x_i=0$ and $pu^x_i=v$, and $u^y_j$ a basis of the $1$-space of $y$.  Then the collection $\{u^x_i, u^y_j\}$ is a basis for the $1$-space of $p$.
\item $p_i=x_i+y_i$, $p_n=x_n+y_n$, $x_iy_i=0$.
\end{enumerate}
\end{proposition}

\begin{proof}
\begin{enumerate}
\item Multiply the relation $pv=(x+y)v=0$ by $x$, and recall that $xy=0$.

\item Multiply the relation $pv=(x+y)v=v$ by $x$, and recall that $xy=0$.

\item Multiply $(x-1)^kv=0$ by $y$ to get $yv=0$.  Then $pv=xv$.  Then $(x-1)^kv=(p-1)^kv=0$.  By the induction hypothesis, $(p-1)^kv=(p-1)v$ implies that $pv=v$, but then $xv=pv=v$, so the result holds.

\item By $(2)$, we have $x^2v=xv$, so in fact, $x^kv=xv=0$.  Then $v=pv=xv+yv=yv$.

\item If $xv=v$, then multiplying by $y$ immediately gives $0=yxv=yv$.  Since $yv=0$, then $pv=(x+y)v=xv=v$.

\item From the previous item, it is clear that the bases $v_x^i$ and $v_y^j$ exist with the desired properties.  All that remains to show is that they form a basis for the $1$-space of $p$.  

Suppose $v$ is in the $1$-space of $p$, so $pv=v$.  Then let $xv=a$ and $yv=b$ so that $pv=(x+y)v=a+b=v$.  
Then $a=xv=x(a+b)=x^2v+xyv=x^2v=xa$.  Then $a$ is in the $1$-space of $x$, and, simlarly, $b$ is in the $1$-space of $y$.  Then the $1$-space of $p$ is spanned by the $1$-spaces of $x$ and $y$, as desired.

\item Let $M_p$, $M_x$ and $M_y$ be matrices for the action of $p$, $x$ and $y$ on $H$.  Then the above results imply that the $0$-eigenspace of $p$ is inherited by $x$ and $y$, and that the $1$-eigenspace of $p$ splits between $x$ and $y$.

We can thus find a basis $\{u^x_k, u^y_l, u^0_m \}$ of $H$ such that: $pu^0_k=xu^0_k=yu^0_k=0$, $xu^x_k=u^x_k$, $pu^x_k=u^x_k$, $yu^x_k=0$, $yu^y_k=u^y_k$, $pu^y_k=u^y_k$, and $xu^y_k=0$.  In this basis, $p$ acts as the identity on $\{u^x_k, u^y_l\}$, and $x$ and $y$ act as orthogonal idempotents.  This proves that $p_i=x_i+y_i$ and $x_iy_i=0$.  Since $p=p_i+p_n=x_i+x_n+y_i+y_n$, then it follows that $p_n=x_n+y_n$.
\end{enumerate}
\end{proof}

\begin{corollary}
There exists a linear basis $v_D^j$ of $\mathbb{C}H_0(S_N)$, indexed by a signed diagram $D$ and some numbers $j$, such that the idempotent $I_D$ obtained from the abstract Jordan decomposition of $C_D$ fixes every $v_D^j$.  For every signed diagram $E\neq D$, the idempotent $I_E$ kills $v_D^j$.
\end{corollary}

The proof of this corollary further shows that this basis respects the branching from $H_0(S_{N-1})$ to $H_0(S_{N})$.  In particular, finding this linear basis for $H_0(S_{N})$ allows the easy recovery of the bases for the indecomposable modules for any $M<N$.

\begin{proof}
Any two sibling idempotents have a linear basis for their $1$-spaces as desired, such that the union of these two bases form a basis for their parent's $1$-space.  Then the union of all such bases gives a basis for the $1$-space of the identity element, which is all of $H$.  

All that remains to show is that for every signed diagram $E\neq D$ with a fixed number of nodes, the idempotent $I_E$ kills $v_D^j$.  Let $F$ be last the common ancestor of $D$ and $E$ under the branching of signed diagrams, so that $F+$ is an ancestor of (or equal to) $D$ and $F-$ is an ancestor of (or equal to) $E$.  Then $I_{F+}$ fixes every $v_D^j$, since the collection $v_D^j$ extends to a basis of the $1$-space of $I_{F+}$.  Likewise, $I_{F-}$ kills every $v_D^j$, by the previous theorem.  
\end{proof}

We now state the main result.  For $D$ a signed diagram, let $D_i$ be the signed sub-diagram consisting of the first $i$ entries of $D$.  

\begin{theorem}
\label{thm:main}
Each diagram demipotent $C_D$ (see Definition \ref{def:ddemipotents}) for $H_0(S_N)$ is demipotent, and yields an idempotent $I_D=C_{D_1}C_{D_2}\cdots C_{D}=C_D^N$.  The collection of these idempotents $\{I_D\}$ form an orthogonal set of primitive idempotents that sum to $1$.
\end{theorem}

\begin{proof}
We can completely determine an element of $\mathbb{C}H_0(S_N)$ by examining its natural action on all of $\mathbb{C}H_0(S_N)$, since if $xv=yv$ for all $v\in \mathbb{C}H_0(S_N)$, then $(x-y)v=0$ for every $v$, and $0$ is the only element of $\mathbb{C}H_0(S_N)$ that kills every element of $\mathbb{C}H_0(S_N)$.

The previous results show that the characteristic polynomial of each diagram demipotent is $X^a(X-1)^b$ for some non-negative integers $a$ and $b$, with all nilpotence associated with the $0$-eigenvalue.  This establishes that the diagram demipotents $C_D$ are actually demipotent, in the sense that there exists some $k$ such that $(C_D)^k$ is idempotent.  Theorem~\ref{thm:algebra} shows that this $k$ grows by at most one with each branching, and thus $k\leq N$.  A prior corollary shows that the idempotents sum to the identity.

The previous corollary establishes a basis for $\mathbb{C}H_0(S_N)$ such that each idempotent $I_D$ either kills or fixes each element of the basis, and that for each $E\neq D$, $I_E$ kills the $1$-space of $I_D$.  Since $I_D$ is in the $1$-space of $I_D$, then $I_E$ must also kill $I_D$.  This shows that the idempotents are orthogonal, and completes the theorem.
\end{proof}

\section{Nilpotence Degree of Diagram Demipotents}
\label{sec:nilp}

Take any $m$ in the $0$-Hecke monoid whose descent set is exactly the set of positive nodes in the signed diagram $D$.  Then $C_Dm=m + (\text{lower order terms})$, by a previous lemma, and $I_Dm= (C_D)^k(m)=m + (\text{lower order terms})$.  The set $\{ I_Dm | Des(m) = \{ \text{positive nodes in D} \}\}$ is thus linearly independent in $H_0(S_N)$, and gives a basis for the projective module corresponding to the idempotent $I_D$.

We have shown that for any diagram demipotent $C_D$, there exists a minimal integer $k$ such that $(C_D)^k$ is idempotent.  Call $k$ the \emph{nilpotence degree} of $C_D$.  The nilpotence degree of all diagram demipotents for $N\leq 7$ is summarized in Figure~\ref{fig:nilpotence}.

\begin{figure}
\centering
\begin{tikzpicture}[scale=.4]

\node (empty) at ( 8,18) [circle,draw] {$1$};

\node (p) at ( -1,15) [circle,draw] {$1$};
\node (m) at ( 17,15) [circle,draw] {$\ldots$};

\draw [-] (empty.south) -- (p.north) node [above,midway]  {$+$};
\draw [-] (empty.south) -- (m.north) node [above,midway]  {$-$};

\node (pp) at ( -9,12) [circle,draw] {$1$};
\node (pm) at (  7,12) [circle,draw] {$1$};

\draw [-] (p.south) -- (pp.north) node [above,midway]  {$+$};
\draw [-] (p.south) -- (pm.north) node [above,midway]  {$-$};

\node (ppp) at (-13,9) [circle,draw] {$1$};
\node (ppm) at ( -5,9) [circle,draw] {$1$};
\node (pmp) at (  3,9) [circle,draw] {$1$};
\node (pmm) at ( 11,9) [circle,draw] {$1$};

\draw [-] (pp.south) -- (ppp.north) node [above,midway]  {$+$};
\draw [-] (pp.south) -- (ppm.north) node [above,midway]  {$-$};
\draw [-] (pm.south) -- (pmp.north) node [above,midway]  {$+$};
\draw [-] (pm.south) -- (pmm.north) node [above,midway]  {$-$};

\node (pppp) at (-15,6) [circle,draw] {$1$};
\node (pppm) at (-11,6) [circle,draw] {$1$};
\node (ppmp) at ( -7,6) [circle,draw] {$1$};
\node (ppmm) at ( -3,6) [circle,draw] {$1$};
\node (pmpp) at (  1,6) [circle,draw] {$2$};
\node (pmpm) at (  5,6) [circle,draw] {$2$};
\node (pmmp) at (  9,6) [circle,draw] {$1$};
\node (pmmm) at ( 13,6) [circle,draw] {$1$};

\draw [-] (ppp.south) -- (pppp.north) node [left,midway]  {$+$};
\draw [-] (ppp.south) -- (pppm.north) node [right,midway]  {$-$};
\draw [-] (ppm.south) -- (ppmp.north) node [left,midway]  {$+$};
\draw [-] (ppm.south) -- (ppmm.north) node [right,midway]  {$-$};
\draw [-] (pmp.south) -- (pmpp.north) node [left,midway]  {$+$};
\draw [-] (pmp.south) -- (pmpm.north) node [right,midway]  {$-$};
\draw [-] (pmm.south) -- (pmmp.north) node [left,midway]  {$+$};
\draw [-] (pmm.south) -- (pmmm.north) node [right,midway]  {$-$};

\node (ppppp) at (-16,3) [circle,draw] {$1$};
\node (ppppm) at (-14,3) [circle,draw] {$1$};
\node (pppmp) at (-12,3) [circle,draw] {$1$};
\node (pppmm) at (-10,3) [circle,draw] {$1$};
\node (ppmpp) at ( -8,3) [circle,draw] {$2$};
\node (ppmpm) at ( -6,3) [circle,draw] {$2$};
\node (ppmmp) at ( -4,3) [circle,draw] {$1$};
\node (ppmmm) at ( -2,3) [circle,draw] {$1$};
\node (pmppp) at (  0,3) [circle,draw] {$2$};
\node (pmppm) at (  2,3) [circle,draw] {$2$};
\node (pmpmp) at (  4,3) [circle,draw] {$2$};
\node (pmpmm) at (  6,3) [circle,draw] {$2$};
\node (pmmpp) at (  8,3) [circle,draw] {$2$};
\node (pmmpm) at ( 10,3) [circle,draw] {$2$};
\node (pmmmp) at ( 12,3) [circle,draw] {$1$};
\node (pmmmm) at ( 14,3) [circle,draw] {$1$};

\draw [-] (pppp.south) -- (ppppp.north) node [left,midway]  {$+$};
\draw [-] (pppm.south) -- (pppmm.north) node [right,midway]  {$-$};
\draw [-] (ppmp.south) -- (ppmpp.north) node [left,midway]  {$+$};
\draw [-] (ppmm.south) -- (ppmmm.north) node [right,midway]  {$-$};
\draw [-] (pmpp.south) -- (pmppp.north) node [left,midway]  {$+$};
\draw [-] (pmpm.south) -- (pmpmm.north) node [right,midway]  {$-$};
\draw [-] (pmmp.south) -- (pmmpp.north) node [left,midway]  {$+$};
\draw [-] (pmmm.south) -- (pmmmm.north) node [right,midway]  {$-$};

\draw [-] (pppm.south) -- (pppmp.north) node [left,midway]  {$+$};
\draw [-] (pppp.south) -- (ppppm.north) node [right,midway]  {$-$};
\draw [-] (ppmm.south) -- (ppmmp.north) node [left,midway]  {$+$};
\draw [-] (ppmp.south) -- (ppmpm.north) node [right,midway]  {$-$};
\draw [-] (pmpm.south) -- (pmpmp.north) node [left,midway]  {$+$};
\draw [-] (pmpp.south) -- (pmppm.north) node [right,midway]  {$-$};
\draw [-] (pmmm.south) -- (pmmmp.north) node [left,midway]  {$+$};
\draw [-] (pmmp.south) -- (pmmpm.north) node [right,midway]  {$-$};
\node (pppppx) at (-16,0) [circle,draw] {$1$};
\node (ppppmx) at (-14,0) [circle,draw] {$1$};
\node (pppmpx) at (-12,0) [circle,draw] {$2$};
\node (pppmmx) at (-10,0) [circle,draw] {$1$};
\node (ppmppx) at ( -8,0) [circle,draw] {$3$};
\node (ppmpmx) at ( -6,0) [circle,draw] {$2$};
\node (ppmmpx) at ( -4,0) [circle,draw] {$2$};
\node (ppmmmx) at ( -2,0) [circle,draw] {$1$};
\node (pmpppx) at (  0,0) [circle,draw] {$2$};
\node (pmppmx) at (  2,0) [circle,draw] {$2$};
\node (pmpmpx) at (  4,0) [circle,draw] {$3$};
\node (pmpmmx) at (  6,0) [circle,draw] {$2$};
\node (pmmppx) at (  8,0) [circle,draw] {$2$};
\node (pmmpmx) at ( 10,0) [circle,draw] {$2$};
\node (pmmmpx) at ( 12,0) [circle,draw] {$2$};
\node (pmmmmx) at ( 14,0) [circle,draw] {$1$};

\draw [-] (ppppp.south) -- (pppppx.north) node [left,midway]  {$\pm$};
\draw [-] (ppppm.south) -- (ppppmx.north) node [left,midway]  {$\pm$};
\draw [-] (pppmp.south) -- (pppmpx.north) node [left,midway]  {$\pm$};
\draw [-] (pppmm.south) -- (pppmmx.north) node [left,midway]  {$\pm$};
\draw [-] (ppmpp.south) -- (ppmppx.north) node [left,midway]  {$\pm$};
\draw [-] (ppmpm.south) -- (ppmpmx.north) node [left,midway]  {$\pm$};
\draw [-] (ppmmp.south) -- (ppmmpx.north) node [left,midway]  {$\pm$};
\draw [-] (ppmmm.south) -- (ppmmmx.north) node [left,midway]  {$\pm$};
\draw [-] (pmppp.south) -- (pmpppx.north) node [left,midway]  {$\pm$};
\draw [-] (pmppm.south) -- (pmppmx.north) node [left,midway]  {$\pm$};
\draw [-] (pmpmp.south) -- (pmpmpx.north) node [left,midway]  {$\pm$};
\draw [-] (pmpmm.south) -- (pmpmmx.north) node [left,midway]  {$\pm$};
\draw [-] (pmmpp.south) -- (pmmppx.north) node [left,midway]  {$\pm$};
\draw [-] (pmmpm.south) -- (pmmpmx.north) node [left,midway]  {$\pm$};
\draw [-] (pmmmp.south) -- (pmmmpx.north) node [left,midway]  {$\pm$};
\draw [-] (pmmmm.south) -- (pmmmmx.north) node [left,midway]  {$\pm$};
\end{tikzpicture}
\caption{Nilpotence degree of diagram demipotents.  The root node denotes the diagram demipotent with empty diagram (the identity).  In all computed example, sibling diagram demipotents have the same nilpotence degree; the lowest row has been abbreviated accordingly for readability.}
\label{fig:nilpotence}
\end{figure}

The diagram demipotent $C_{+\cdots +}$ with all nodes positive is given by the longest word in the $0$-Hecke monoid, and is thus already idempotent.  The same is true of the diagram demipotent $C_{-\cdots -}$ with all nodes negative.  As such, both of these elements have nilpotence degree $1$.

\begin{lemma}
The nilpotence degree of sibling diagram demipotents $C_{D+}$ and $C_{D-}$ are either equal to or one greater than the nilpotence degree $k$ of the parent $C_{D}$.  Furthermore, the nilpotence degree of sibling diagram demipotents are equal.
\end{lemma}

\begin{proof}
Let $x$ and $y$ be the sibling diagram demipotents, with parent diagram demipotent $p$, so $p=C_{D}=L_DR_D, x=C_{D+}=L_D\pi_NR_D, y=C_{D-}=L_D(1-\pi_N)R_D$.  Let $p$ have nilpotence degree $k$, so that $p^k=p^{k+1}$.  We have already seen that the nilpotence degree of $x$ and $y$ is at most $k+1$.  We first show that the nilpotence degree of $x$ or $y$ cannot be less than the nilpotence degree of $p$.  

Recall the following quotients of $\mathbb{C}H_0(S_N)$: 
\begin{eqnarray*}
\Phi_N^+ &:& \mathbb{C}H_0(S_N) \rightarrow \mathbb{C}H_0(S_{N-1}) \\
\Phi_N^+(\pi_i) &=&     \begin{cases}
      1          & \text{if $i=N$,}\\
      \pi_i & \text{if $i \neq N$.}
    \end{cases}\\
\Phi_N^- &:& \mathbb{C}H_0(S_N) \rightarrow \mathbb{C}H_0(S_{N-1}) \\
\Phi_N^-(\pi_i) &=&     \begin{cases}
      0          & \text{if $i=N$,}\\
      \pi_i & \text{if $i \neq N$.}
    \end{cases}
\end{eqnarray*}
given by introducing the relation $\pi_N=1$.  One can easily check that these are both morphisms of algebras.  
Notice that $\Phi_N^+(x)=p$, and $\Phi_N^-(y)=p$.  Then if the nilpotence degree of $x$ is $l<k$, we have 
$p^l=\Phi_N^+(x^l)=\Phi_N^+(x^{l+1})=p^{l+1}$, implying that the nilpotence degree of $p$ was actually $l$, a contradiction.  The same argument can be applied to $y$ using the quotient $\Phi_n^-$.

Suppose one of $x$ and $y$ has nilpotence degree $k$. Assume it is $x$ without loss of generality.  Then:
\begin{eqnarray*}
p^k & = & p^{k+1} \\
\Leftrightarrow x^k+y^k & = & x^{k+1}+y^{k+1} \\
\Leftrightarrow x^{k+1}+y^k & = & x^{k+1}+y^{k+1} \\
\Leftrightarrow y^k & = & y^{k+1} \\
\end{eqnarray*}  
Then the nilpotence degree of $y$ is also $k$.

Finally, if neither $x$ nor $y$ have nilpotence degree $k$, then they both must have nilpotence degree $k+1$.
\end{proof}

Computer exploration suggests that siblings always have equal nilpotence degree, and that nilpotence degree either stays the same or increases by one after each branching.

\begin{lemma}
Let $D$ be a signed diagram with a single sign change, or the sibling of such a diagram.  Then $C_D$ is idempotent (and thus has nilpotence degree $1$).
\end{lemma}

\begin{proof}
We prove the statement for a diagram with single sign change, since siblings automatically have the same nilpotence degree.
Without loss of generality let the diagram of $D$ be 
$--\cdots--++\cdots++$.  Let $L$ the subset of the index set with negative marks in $D$.  Let $i$ be the minimal element of the index set with a positive mark, and let $H=I \setminus L, i$.  Then:
\[
C_D= w_L^- w_H^+ \pi_i w_H^+ w_L^-.
\]
Notice that $w_H^+$ and $w_L^-$ commute.

Set $y=w_L^- w_H^+ (1-\pi_i) w_H^+ w_L^-$, 
and $p = C_D+y = w_L^- w_H^+ w_H^+ w_L^- = w_H^+ w_L^-$.

Now $y$ is not a diagram demipotent, though $p$ could be considered a diagram demipotent for disconnected Dynkin Diagram with the $i$th node removed.

It is immediate that:
\[
p^2=p, \qquad C_Dp=C_D=pC_D \qquad yp=y=py
\]

Now we can establish orthogonality of $C_D$ and $y$:
\begin{eqnarray*}
C_Dy & = & (w_L^- w_H^+ \pi_i w_H^+ w_L^-) (w_L^- w_H^+ (1-\pi_i) w_H^+ w_L^-) \\
& = & w_L^-( w_H^+ \pi_i w_H^+)(w_L^- (1-\pi_i)w_L^-) w_H^+\\
& = & w_L^- w_{H\cup i}^+ w_{L\cup i}^- w_H^+ \\
& = & 0
\end{eqnarray*}
The product of $w_{H\cup i}^+$ and $w_{L\cup i}^-$ is zero, since $w_{H\cup i}^+$ has a $\pi_i$ descent, and 
$w_{L\cup i}^-$ has a $\barp_i$ descent.

Then $C_D = pC_D = (C_D + y)C_D = (C_D)^2$, so we see that $C_D$ is idempotent.
\end{proof}

In particular, this lemma is enough to see why there is no nilpotence before $N=5$; every signed Dynkin diagrams with three or fewer nodes has no sign change, one sign change, or is the sibling of a diagram with one sign change.

\begin{proposition}
Let $D$ be any signed diagram with $n$ nodes, and let $E$ be the largest prefix diagram such that $E$ has a single sign change, or is the sibling of a diagram with a single sign change.  Then if $E$ has $k$ nodes, the nilpotence degree of $D$ is at most $n-k$.
\end{proposition}

\begin{proof}
This result follows directly from the previous lemma and the fact that the nilpotence degree can increase by at most one with each branching.
\end{proof}

This bound is not quite sharp for $H_0(S_N)$ with $N\leq 7$:  The diagrams $+-++$, $+-+++$, and $+-++++$ all have nilpotence degree $2$.  However, at $N=7$, the highest expected nilpotence degree is $3$ (since every diagram demipotent with three or fewer nodes is idempotent), and this degree is attained by $4$ of the demipotents.  These diagram demipotents are $++-+++$, $+-+-++$, and their siblings.

An open problem is to find a formula for the nilpotence degree directly in terms of the diagram of a demipotent.

\section{Further Directions}
\label{sec:quest}

\subsection{Conjectural Demipotents with Simpler Expression}

Computer exploration has suggested a collection of demipotents that are simpler to describe than those we have presented here.  

For a word $w=w_1w_2\cdots w_k$ and a signed diagram $D$, we obtain the \emph{masked word} $w^D$ by applying the sign of $i$ in $D$ to each instance of $i$ in $w$.  For example, for the word $w=(1,2,1,3,1,2)$ and $D=+-+$, the masked word is $w^D=(1,-2,1,3,1,-2)$.  A masked word yields an element of $H_0(S_N)$ in the obvious way: we write 
\[
\pi_w^D:=\prod \pi_{w_i}^{sgn(i)},
\]
where $sgn(i)$ is the sign of $i$ in $D$.  

Some masked words are demipotent and otherare not.  We call a word \emph{universal} if:
\begin{itemize}
\item $w$ contains every letter in $I$ at least once, and
\item $w^D$ is demipotent for every signed diagram $D$.
\end{itemize} 

\begin{conjecture}
\label{conj:unverisality}
The word $u_N=(1,2,\ldots, N-2, N-1, N-2, \ldots, 2,1)$ is universal.
\end{conjecture}

Computer exploration has shown that $u_N$ are universal up to $\mathbb{C}H_0(S_9)$, and that the idempotents thus obtained are the same as the idempotents obtained from the diagram demipotents $C_D$.  However, these demipotents $u^D_N$, though they branch in the same way as the diagram demipotents, fail to have the sibling rivalry property.  Thus, another method should be found to show that these elements are demipotent.

An important quotient of the zero-Hecke monoid is the monoid of \emph{Non-Decreasing Parking Functions}, $NDPF_N$.  These are the functions $f: [N] \rightarrow [N]$ satisfying
\begin{itemize}
\item $f(i)\leq i$, and 
\item For any $i\leq j$, then $f(i)\leq f(j)$.
\end{itemize}
This monoid can be obtained from $H_0(S_N)$ by introducing the additional relation:
\[
\pi_i\pi_{i+1}\pi_i = \pi_i\pi_{i+1}.
\]
The lattice of idempotents of the monoid $NDPF_N$ is identical to the lattice of idempotents in $H_0(S_N)$.  We have shown that every masked word $u_N^D$ is idempotent in the algebra of $NDPF_N$, supporting Conjecture~\ref{conj:universality}.  For the full exploration of $NDPF_N$, including the proof of the claim that $u_N^D$ is idempotent in $\mathbb{C}NDPF_N$, see \cite{Denton_Hivert_Schilling_Thiery.JTrivialMonoids}.

\subsection{Direct Description of the Idempotents}

A number of questions remain concerning the idempotents we have constructed.

First, uniqueness of the idempotents described in this paper is unknown.  In fact, there are many families of orthogonal idempotents in $H_0(S_N)$.  The idempotents we have constructed are invariant as a set under the automorphism $\Psi$, and compatible with the branching from $S_{N-1}$ to $S_N$ according to the choice of orientation of the Dynkin diagram.  

Second, computer exploration has shown that, over the complex numbers, the idempotents obtained from the diagram demipotents have $\pm 1$ coefficients.  This phenomenon has been observed up to $N=9$.  This seems to be peculiar to the construction we have presented, as we have found other idempotents that do not have this property.  It would be interesting to have an even more direct construction of the idempotents, such as a rule for directly determining the coefficients of each idempotent.

It should be noted that a general `lifting' construction has long been known, which constructs orthogonal idempotents in the algebra.  (See~\cite[Chapter 77]{curtis_reiner.1962})  A particular implementation of this lifting construction for algebras of $\mathcal{J}$-trivial monoids is given in [DHST, FORTHCOMING].  This lifting construction starts with the idempotents in the monoid, which in the semisimple quotient have the multiplicative structure of a lattice.  In the case of a zero-Hecke algebra with index set $I$, these idempotents are just the long elements $w_J^+$, for any $J\subset I$.  Then the multiplication rule in the semisimple quotient for two such idempotents $w_J^+$, $w_K^+$ is just $w_K^+w_J^+=w_{J\cup K}^+$.  Each idempotent in the semisimple quotient is in turn lifted to an idempotent in the algebra, and forced to be orthogonal to all idempotents previously lifted.  Many sets of orthogonal idempotents can be thus obtained, but the process affords little understanding of the combinatorics of the underlying monoid.  

The $\pm 1$ coefficients that have been observed in the idempotents thus far constructed suggest that there are still interesting combinatorics to be learned from this problem.

\subsection{Generalization to Other Types}

A combinatorial construction for idempotents in the zero-Hecke algebra for general Coxeter groups would be desirable.  It is simple to construct idempotents for any rank 2 Dynkin diagram.  The author has also constructed idempotents for type $B_3$ and $D_4$, but has not been able to find a satisfactory formula for general type $B_N$ or $D_N$.  

A major obstruction to the direct application of our construction to other types arises from our expressions for the longest elements in type $A_N$.  For the index set $J \cup \{k\} \subset I$, where $k$ is larger (or smaller) than any index in $J$ we have expressed the longest element for $J\cup \{\pi_k\}$ as $w_J^+\pi_k w_J^+$.  This expression contains only a single $\pi_k$.  In every other type, expressions for the longest element generally require at least two of any generator corresponding to a leaf of the Dynkin diagram.  This creates an obstruction to branching demipotents in the way we have described for type $A_N$.

For example, in type $D_4$, a reduced expression for the longest element is $\pi_{423124123121}$.  The generators corresponding to leaves in the Dynkin diagram are $\pi_1$, $\pi_3$, and $\pi_4$, all of which appear at least twice in this expression.  (In fact, this is true for any of the $2316$ reduced words for the longest element in $D_4$.)  Ideally, to branch easily from type $A_3$, we would be able to write the long element in the form $w_J^+\pi_4w_J^+$, where $4 \not \in J$, but this is clearly not possible.

\bibliographystyle{abbrvnat}
\bibliography{tom-fpsac}
\label{sec:biblio}

\end{document}